\documentclass[a4paper,11pt]{amsart}
\usepackage{amsmath,amsthm,amssymb}
\usepackage{amsfonts}
\usepackage{txfonts}
\usepackage{latexsym}
\usepackage{iftex}
\ifpTeX
\usepackage[dvipdfmx]{graphics,color}
\else
\usepackage{graphicx,color}
\fi
\usepackage{eucal}
\usepackage{mathrsfs}
\ifpTeX
\usepackage[dvipdfmx]{hyperref}
\usepackage{pxjahyper}
\else
\usepackage{hyperref}
\fi
\usepackage{tikz-cd}

\theoremstyle{plain}
\newtheorem{theorem}{Theorem}[section]
\newtheorem*{thm*}{Theorem}
\newtheorem{althm}{Theorem}

\newtheorem{proposition}[theorem]{Proposition}

\newtheorem{lemma}[theorem]{Lemma}
\newtheorem{corollary}[theorem]{Corollary}

\theoremstyle{definition}
\newtheorem{definition}[theorem]{Definition}
\newtheorem{example}[theorem]{Example}
\makeatletter
\renewenvironment{proof}[1][\proofname]{\par
  \normalfont
  \topsep6\p@\@plus6\p@ \trivlist
  \item[\hskip\labelsep{\bfseries #1}\@addpunct{\bfseries.}]\ignorespaces
}{%
  \endtrivlist
}
\renewcommand{\proofname}{proof}
\theoremstyle{remark}
\newtheorem{remark}[theorem]{Remark}

\numberwithin{equation}{section}

\title{Almost mathematics of pointed symmetric monoidal model categories by Smith ideal theory}
\date{\today} 
\author{Yuki Kato}
\thanks{The author was supported by Grants-in-Aid for Scientific
 Research No.~23K030080, Japan Society for the Promotion of Science.}
\address{National institute of technology, Kurume college, 
	      1-1-1, Komorino, Kurume, Fukuoka, JAPAN 830-8555.}
\email{kato\_051@kurume-nct.ac.jp}

\newcommand{\DR}{\mathbb{R}}
\newcommand{\DL}{\mathbb{L}}

\newcommand{\Hom}{\mathrm{Hom}}

\newcommand{\Map}{\mathrm{Map}}
\newcommand{\Alg}{\mathrm{Alg}}
\newcommand{\Mod}{\mathrm{Mod}}

\newcommand{\CAlg}{\mathrm{CAlg}}

\newcommand{\tm}{\tilde{\mathfrak{m}}}
\newcommand{\Fun}{\mathrm{Fun}}

\newcommand{\Ch}{\operatorname{Ch}}

\usepackage{enumerate}
\usepackage{array}
\usepackage[all]{xy} 
\usepackage{delarray}
\def\qed{{\hfill $\Box$}}
\begin{document}
\thispagestyle{empty}
\begin{abstract}
This article is a generalization of a result in Quillen's note
\cite{Q} ``Module theory over non-unital rings'' giving a one-to-one
correspondence between bilocalization of abelian categories of modules
and idempotent ideals of the base ring. Faltings~\cite{Faltings}; Gabber
and Ramero~\cite{GR} established almost mathematics, the same as
Quillen's bilocalization of a category of modules by nil-modules. 

In this paper, by using the theory of Smith ideals mentioned in
Hovey~\cite{Smith-ideals}, we consider almost mathematics of symmetric
monoidal pointed model categories and prove a weak analogue of the
one-to-one correspondence in Quillen's note~\cite{Q}.
\end{abstract}

\maketitle

\section{Introduction}
\label{sec-Introduction}
Faltings~\cite{Faltings} first introduced the theory of almost mathematics to apply $p$-adic Hodge theory, and subsequently, Gabber and
Ramero~\cite{GR} developed the theory of algebra within this framework. In their terminology, the theory includes almost ring
theory, almost modules, almost algebras, and almost homotopy
algebras. Building on Gabber and Ramero's work, almost mathematics is
formulated by localizing an abelian category of modules at almost
isomorphisms: let $V$ be a unital commutative ring with an idempotent
ideal $\mathfrak{m}$ of $V$. A $V$-module $M$ is defined to be {\it almost
zero} if it is annihilated by $\mathfrak{m}$, and almost mathematics is an
algebraic theory constructed by localizing the category of $V$-modules with
respect to the Serre subcategory of almost zero modules. {\it Almost isomorphisms}
are morphisms whose kernels and cokernels are both almost zero. In the
localized category, module objects are referred to {\it as almost modules} ,
while algebra objects are called {\it almost algebras}.

Independently, in an unpublished note~\cite{Q}, Quillen considered the
same theory of almost mathematics as linear algebra over non-unital
rings. He characterized it as a bilocalization, that is, both a
localization and a colocalization, of abelian categories of
modules. Further, Quillen proved the following one-to-one correspondence
between idempotent ideals and reflexive bilocalization functors.
\begin{theorem}(\cite{Q} Proposition 6.5)
\label{QuillenTh}
Let $V$ be a commutative ring with a multiplicative unit. There is a
one-to-one correspondence between idempotent ideals of $V$ and Serre
subcategories $\mathcal{S}$ of $\Mod_V$ for which the localization
$F:\Mod_V \to \Mod_V/\mathcal{S} $ is also a colocalization. \qed
\end{theorem}

The purpose of this article is to establish a theory of almost
mathematics for symmetric monoidal model categories and to prove a
monoidal categorical analogue of Theorem~\ref{QuillenTh}. Namely, using
the commutative Smith ideal theory developed by
Hovey~\cite{Smith-ideals} and White--Yau~\cite{WY2024}, we introduce
homotopy analogues of idempotent ideals and almost isomorphisms, which
we refer to as homotopically idempotent Smith ideals and almost weak
equivalences. The main result of this paper are summarized in
Theorem~\ref{theoremA}.
\begin{althm}
Let $\mathcal{M}$ be a pointed symmetric monoidal model category with a monoidal unit object $V$, and let $j: I \to V$ be a homotopically idempotent commutative Smith ideal. Then the internal-hom functor
\[
 \Map_{\mathcal{M}}(\tilde{I},\, -):  \mathcal{M}      \to    \mathcal{M}
\]
is equivalent to the Bousfield bilocalization functor with respect to
almost weak equivalences.
\end{althm}
Since $\Map_{\mathcal{M}}(\tilde{I},\, -)$ is a right Quillen
functor, this theorem shows that the Bousfield localization with respect
to almost weak equivalences is both a left and a right Quillen functor.

Conversely, in the case that the monoidal unit generates the stable
symmetric monoidal model category in the sense of MacLane~\cite[Chapter
V, Section 7, p.127]{zbMATH00195199}, we can construct a homotopically
idempotent ideal (Proposition~\ref{theorem-idem} and
Theorem~\ref{theoremB}).

Furthermore, we generalize some results on almost algebras to the
model category setting using the language of Smith ideals. Again, let
$V$ be a
commutative ring with an idempotent ideal $\mathfrak{m}$. We assume that
$\tm$ is $V$-flat.  Let $\CAlg(\Mod_{V})$ denote the category of
$V$-algebras and $\CAlg(\Mod_{V}^{\rm al})$ the category of almost
algebras. Gabber and Ramero defined the functor
 \[
    (-)_{!!}: \CAlg(\Mod_{V}) \to  \CAlg(\Mod_{V}^{\rm al})
\]
 by setting \[
A_{!!} = V \amalg_{ \tm } (\tm \otimes_V A)
\]
 for any almost unital $V$-algebra $A$ in \cite[Section 2.2]{GR}. Those
morphisms $\tm \to V \to A $ and $\tm \to \tm \otimes_V A \to A$ induce
a ring homomorphism $A_{!!} \to A$, which is an almost isomorphism. Due to
\cite[Proposition 2.2.29]{GR}, the almost unitalization $(-)_{!!}$ is a
left adjoint of the localization $ \CAlg(\Mod_{V}) \to
\CAlg(\Mod_{V}^{\rm al})$. We call the left adjoint $(-)_{!!}$ {\it
almost unitalization} in this paper.  Using Smith ideal theory, we
generalize the almost unitalization functor $(-)_{!!}$ in the language
of Smith ideals.


In Section~\ref{sec:almost}, we introduce the theory of
almost mathematics of pointed symmetric monoidal model categories;
homotopically idempotent commutative Smith ideals, homotopically almost
zero objects, and homotopically almost weak equivalences.

In Section~\ref{sec:almost-alg}, we generalize the above almost
unitalization functor in the language of Smith ideal theory, and we
prove that it is a left adjoint of the Bousfield localization functor
with respect to the class of almost weak equivalences defined in
Section~\ref{sec:almost}; see Theorem~\ref{adjoint} and
Corollary~\ref{app-almost}.

In Section~\ref{sec:bilocal}, as the inverse direction of
Theorem~\ref{theoremA}, we construct a homotopically idempotent Smith
ideal from a given symmetric monoidal bilocalization functor on a
symmetric monoidal model category under mild conditions. The result is
proved as Theorem~\ref{theoremB}.

Finally, in Section~\ref{sec:non-unital}, we remark that the
cokernel-and-kernel functors in the (operadic) Smith ideal theory
induce a categorical equivalence between non-unital algebra objects and
augmented algebra objects of stable symmetric monoidal model categories.

For the reader's convenience, we mention related independent work of
Hebestreit and Scholze on higher almost ring
theory~\cite{hebestreit2025notehigherringtheory}. Their work develops a
higher almost mathematics framework in an $\infty$-categorical setting.
The present paper is complementary, developing a model-categorical
framework based on Smith ideal techniques. To clarify the independent
development of the two projects, we also record that the fourth version of the present preprint appeared as
arXiv:2301.04541v4 on July 29, 2024 UTC (July 30 in Japan), and that the first version of
Hebestreit and Scholze's preprint appeared as arXiv:2409.01940v1 on
September 3, 2024.

\subsection{Example: The almost mathematics of a stable model category of chain complexes}
\label{sec:Example}
In this section, we consider the model category of
chain complexes of modules as an example of almost mathematics of model
categories.  We fix a commutative base ring $V$ and an idempotent ideal
$\mathfrak{m}$ ($ \mathfrak{m}^2 = \mathfrak{m}$) of $V$, and we assume that
the tensor product $\tm= \mathfrak{m} \otimes_V \mathfrak{m}$ is a flat
$V$-module. Let $\Ch(V)$ denote the category of chain complexes of
$V$-modules.  Then $\Ch(V)$ has a canonical model structure called the
{\it injective model structure} that cofibrations are degree-wise
injections and weak equivalences are quasi-isomorphisms.  Since we
assume that $\tm$ is a flat $V$-module, the functor $ \tm \otimes_V
(-):\Ch(V) \to \Ch(V) $ is exact and a left Quillen functor, whose right
adjoint is the mapping complex functor $\Map_V(\tm ,\, -)$.

\begin{example}
A perfectoid commutative valuation ring is a valuation ring of a
perfectoid field in the sense of Scholze~\cite[p.246, Definition
1.2]{Sh12}.  Let $V$ be a perfectoid valuation ring. Then the set of
topologically nilpotent elements $\mathfrak{m}$ of $V$ is an idempotent
maximal ideal and a flat $V$-module, as explained, for example, in
Bhatt's lecture note~\cite[p.18, Example 4.1.3]{Bhatt-lecture}. In
general, the product map $\tm \to
\mathfrak{m}$ is an isomorphism if $\mathfrak{m}$ is a flat $V$-module.

\end{example}

In the category $\mathrm{Ch}(V)$, almost weak equivalences are defined
as follows:
\begin{definition}
A morphism $f: M \to N$ of $V$-complexes is said to be an {\it almost
quasi-isomorphism} if the induced morphism $\tm \otimes_V f: \tm
\otimes_V M \to\tm \otimes_V N $ is a quasi-isomorphism. Let
$\Ch(V)^{\rm al}$ denote the Bousfield localization of the model
category $\Ch(V)$ by the class of almost quasi-isomorphisms. A chain
complex $E$ is said to be
{\it almost local} if $E$ is a local object with respect to almost
quasi-isomorphisms. That is, for any almost quasi-isomorphism $f: M \to
N$, the induced map
\[
 f^*:    \Hom_{D(V)}( N ,\,   E  ) \to  \Hom_{D(V)}( M ,\,  E )
\]
is an isomorphism in the derived category $D(V)$ of $\Ch(V)$.
\end{definition}

In Quillen's note~\cite{Q}, a $V$-module $E$ is said to be {\it closed } if the induced
homomorphism $E \to \Hom_{V}(\mathfrak{m} ,\,E)$ is an isomorphism, and
it is shown that closed modules are almost local objects and vice versa;
see \cite[Proposition 5.3]{Q}.  The following lemma characterizes almost
local objects.
\begin{proposition}
\label{closed-local2}
Let $E$ be a chain complex of $V$-modules. Then $E$ is almost local if
and only if the product $\mu_E:\tm \otimes_V E \to E $ induces a
quasi-isomorphism
\[
 E \to  \Map_{V}(\tm,\, E).
\]
\end{proposition}
\begin{proof}
The only if direction is clear by definition of almost local objects.

Conversely, assume that $E$ is almost local. Since the canonical
morphism $\tm \otimes_V M \to M$ is an almost quasi-isomorphism for any
$M \in \Ch(V)$, the induced morphism
\[
 \Hom_{D(V)}(M,\, E  ) \to      \Hom_{D(V)}(M  ,\, \Map_V ( \tm ,\,  E)  )  \simeq  \Hom_{D(V)}(\tm \otimes_V  M ,\,  E  ).
\]
is an isomorphism. By the Yoneda lemma, $E \to \Map_V ( \tm ,\,  E)  $ is a quasi-isomorphism.
\qed
\end{proof}

\begin{theorem}
\label{Example}
The mapping complex functor
\[
 \Map_{V}(\tm ,\, -): \Ch(V) \to \Ch(V)^{\rm al}
\]
is equivalent to the localization functor $L^{\rm al}: \Ch(V) \to
\Ch(V)^{\rm al}$ on the derived categories.
\end{theorem}
\begin{proof}
This is the special case of Theorem~\ref{theoremA}. \qed
\end{proof}

\section{Almost mathematics of pointed symmetric monoidal model categories}
\label{sec:almost}
In this section, following White and Yau's work~\cite{WY2024}, we recall
the definition and some properties of symmetric monoidal model
categories and their (commutative) Smith ideals.  Next, in order to
establish almost mathematics, we define homotopically idempotent Smith
ideals, and we prove the main result, Theorem~\ref{theoremA}.

A category is {\it pointed} if it has an initial object $\emptyset$ and
a final object $*$ such that the unique morphism $\emptyset \to *$ is an
isomorphism. We call the initial object the {\it zero object}, and $0$
denotes it. Throughout this paper, we consider pointed (model)
categories.

Let $\mathcal{M}$ be a model category. For any object $M \in
\mathcal{M}$, let $Q(M) \to M$ denote a cofibrant replacement and $M \to
R(M)$ a fibrant replacement. It is well-known that the following
adjunctions
\[
\xymatrix@1{
     \mathcal{M}^{\rm cof} \ar[r]<0.5mm>  & \mathcal{M}  \ar[r]<0.5mm>^R \ar[l]<0.5mm>^Q \ar[l]<0.5mm>  &  \mathcal{M}^{\rm fib}  \ar[l]<0.5mm>
}
\]
are Quillen equivalences, where $ \mathcal{M}^{\rm cof}$
(resp. $\mathcal{M}^{\rm fib}$) denotes the full subcategory spanned by
all cofibrant (resp. fibrant) objects.

\subsection{Commutative Smith ideals of symmetric monoidal model categories}

A symmetric monoidal model category $\mathcal{M}$ is a model category
with a symmetric monoidal structure $ -\otimes - : \mathcal{M} \times
\mathcal{M} \to \mathcal{M}$ satisfying the following two axioms:
\begin{enumerate}[(1)]
 \item Unit axiom: The monoidal unit $V$ of $\mathcal{M}$ satisfies the
       condition that, for any cofibrant object $M$ and cofibrant
       replacement of $Q(V) \to V$, the induced morphisms $ Q(V) \otimes
       M \to V \otimes M \to M $ and $
M \otimes Q(V) \to M \otimes V \to M $ are weak equivalences.
 \item Pushout-product axiom: For any cofibrations $f:X_0 \to X_1$ and
       $g:Y_0 \to Y_1$, the pushout product $f \Box g :(X_0 \otimes Y_1)
       \amalg_{X_0 \otimes Y_0} (X_1 \otimes Y_0) \to X_1 \otimes Y_1 $
       is a cofibration. Further, if $f$ or $g$ is a trivial
       cofibration, so is $f \Box g$.
\end{enumerate}

We say that the symmetric monoidal structure on $\mathcal{M}$ is {\it closed} if $\mathcal{M}$ contains an internal hom-object
$\Map_\mathcal{M}(M,\,N)$ with respect to the monoidal structure $ -
\otimes - : \mathcal{M} \times \mathcal{M} \to \mathcal{M}$ for any
$M,\,N \in \mathcal{M}$.  In this case, the push-out product axiom
implies that, for any cofibrant object $M$, both induced functors
$- \otimes M: \mathcal{M} \to \mathcal{M}$ and $M \otimes -: \mathcal{M}
\to \mathcal{M}$ are left Quillen functors.

Further, we assume that all model categories in this paper are {\it
cofibrantly generated}.  For any class $\mathcal{T}$ of morphisms in a
category, ${}^\boxslash \mathcal{T}$ (resp. $\mathcal{T}^\boxslash$)
denotes the class of morphisms having the right (resp. left) lifting
property with respect to all morphisms in $\mathcal{T}$. Then
${}^\boxslash(\mathcal{T}^\boxslash)$ is said to be the {\it weakly
saturated class} of morphisms generated by $\mathcal{T}$. In this paper,
we always assume that all model categories are cofibrantly generated:
There exist small subsets $\mathcal{I}$ and $\mathcal{J}$ of the set of
morphisms of the model category such that the collection of all
cofibrations (resp. trivial cofibrations) is the weakly saturated class
of morphisms generated by $\mathcal{I}$ (resp. $\mathcal{J}$). In this
paper, we call $\mathcal{I}$ (resp. $\mathcal{J}$) {\it the set of
generating cofibrations} (resp. {\it the set of generating trivial
cofibrations}).  Further, if a cofibrantly generated model category is
locally presentable, it is said to be {\it combinatorial}. By the
theorem of Barwick and Smith~\cite[Theorem 4.7]{MR2771591}, for any left
proper and combinatorial model category and a small set of homotopy
classes of morphisms in it, the Bousfield localization exists such that
the localized model category is also left proper and combinatorial.

Let $\Delta^1$ denote the category with two objects $0$ and $1$, and
only one non-trivial morphism $0 \to 1$. The functor category
$\Fun(\Delta^1,\,\mathcal{M})$ from $\Delta^1$ to $\mathcal{M}$ is
called the {\it arrow category} of $\mathcal{M}$ and is denoted by
$\mathrm{Ar}(\mathcal{M})$.

By Hovey~\cite[Theorem 1.2]{Smith-ideals}, the arrow category
$\mathrm{Ar}(\mathcal{M})$ has two distinct symmetric monoidal
structures derived from that of $\mathcal{M}$. For any two morphisms
$f:X_0 \to X_1$ and $g: Y_0 \to Y_1$ in $\mathcal{M}$, one has a
commutative square:
\[
 \xymatrix@1{ X_0 \otimes Y_0 \ar[r]^{f \otimes \mathrm{id}}
\ar[d]_{\mathrm{id}\otimes g} & X_1 \otimes Y_0
\ar[d]^{\mathrm{id}\otimes g}\\
X_0 \otimes Y_1 \ar[r]_{f \otimes \mathrm{id}}& X_1 \otimes Y_1.
}
\]
The {\it diagonal (or tensor) product monoidal structure} is defined by
$f \otimes g$ as the composition
\[
f \otimes g = (f \otimes \mathrm{id}) \circ (
\mathrm{id} \otimes g)= ( \mathrm{id} \otimes g)\circ (f \otimes
\mathrm{id}) : X_0 \otimes Y_0 \to X_1 \otimes Y_1.
\]
The other is the {\it push-out
product monoidal structure} defined by the induced morphism:
\[
 f \Box g: ( X_0 \otimes Y_1   ) \amalg_{ X_0 \otimes Y_0 } (   X_1 \otimes Y_0 ) \to  X_1 \otimes Y_1.
\]

The arrow category $\mathrm{Ar}(\mathcal{M})$ has the evaluation
functors $\mathrm{Ev}_{i}(f: X_0 \to X_1)=X_i \ (i=0,\,1)$, which have a left
adjoint and a right adjoint:
\[
\xymatrix@1{
\mathrm{Ev}_{i} : \mathrm{Ar}(\mathcal{M})  \ar[r]  &
 \ar[l]<-1mm> \ar[l]<1mm>   \mathcal{M}:L_i ,\,U_i \ (i=0,\,1).
}
\]
Namely, those adjoint functors are defined by $L_0(X)=( \mathrm{id}:X
\to X )$, $U_0(X)= (X \to * )$, $L_1(X)=( \emptyset \to X )$, and
$U_1(X)= ( \mathrm{id}: X \to X )$. Further, those left adjoints $L_i :
\mathcal{M} \to \mathrm{Ar}^\Box(\mathcal{M}) \ (i=0,1)$ are left
Quillen and strict monoidal functors, and those right adjoints $U_i
:\mathcal{M} \to \mathrm{Ar}^\otimes(\mathcal{M}) \ (i=0,1)$ are right
Quillen functors.

Let $\mathrm{Ar}^\Box(\mathcal{M})$ (resp. $\mathrm{Ar}^\otimes(\mathcal{M})$) denote the symmetric monoidal category with respect to the push-out (resp. tensor) product structure whose monoidal unit is $L_1(V)$ (resp. $U_0(V)$), where $V$ is the monoidal unit of $\mathcal{M}$.  Then $\mathrm{Ar}^\Box(\mathcal{M})$
(resp. $\mathrm{Ar}^\otimes(\mathcal{M})$) equips a symmetric monoidal
model structure with respect to the projective (resp. injective) model
structure, which is defined by weak equivalences to be degree-wise weak
equivalences and fibrations (resp. cofibrations) to be degree-wise
fibrations (resp. degree-wise cofibrations).

These monoidal model categories $\mathrm{Ar}^\Box(\mathcal{M})$ and $\mathrm{Ar}^\otimes(\mathcal{M})$ are connected by the cokernel-and-kernel adjunction
\[
 \mathrm{cok} : \mathrm{Ar}^\Box(\mathcal{M}) \rightleftarrows \mathrm{Ar}^\otimes(\mathcal{M}) :\mathrm{ker},
\]
 which is a Quillen adjunction. The cokernel functor $\mathrm{cok}$ is defined by
\begin{align*}
  \mathrm{cok}: \mathrm{Ar}^\Box(\mathcal{M}) &\to
  \mathrm{Ar}^\otimes(\mathcal{M}) \\ (f :X \to Y) &\mapsto (Y \to
  \mathrm{Coker}(f)),
\end{align*}
and the kernel functor  is defined by
\begin{align*}
  \mathrm{ker}: \mathrm{Ar}^\otimes(\mathcal{M}) &\to \mathrm{Ar}^\Box(\mathcal{M}) \\
  (f :X \to Y) &\mapsto (\mathrm{Ker}(f) \to X),
\end{align*}
respectively. In the case that $\mathcal{M}$ is a pointed, left proper, and
combinatorial symmetric monoidal model category, the adjunction
is lax monoidal Quillen by Hovey's work (the proof is available at
\cite[Proposition 2.4.3]{WY2024}).  Further, if $\mathcal{M}$ is stable, in
addition, the adjunction $(\mathrm{cok},\,\mathrm{ker})$ is a Quillen
equivalence. (See Hovey~\cite[Theorem 1.4]{Smith-ideals} or White and
Yau~\cite[Proposition 2.3]{WY2024}).

We recall the definition of commutative Smith ideals. A commutative
monoid object $A$ of $\mathcal{M}$ is a monoid whose product $\mu_A: A
\otimes A \to A$ commutes with the braiding $ \tau_{A,\,A}: A \otimes A
\to A \otimes A$. We denote by $\CAlg(\mathcal{M})$ the subcategory of
commutative objects of $\mathcal{M}$.

\begin{definition}[{\rm c.f.} \cite{WY2024} Definition 3.3.1]
Let $\mathcal{M}$ be a pointed symmetric monoidal model category. A {\it
commutative Smith ideal} is a commutative monoid object of the symmetric
monoidal model category $\mathrm{Ar}^\Box(\mathcal{M})$.
\end{definition}
\subsection{Homotopically idempotent commutative Smith ideals}

We introduce a homotopy analogue of idempotent ideals as follows:
\begin{definition}
\label{idem-def} Let $\mathcal{M}$ be a pointed symmetric monoidal model
category and $V$ a monoidal unit object. A commutative Smith ideal $j:I \to V$ is {\it homotopically idempotent} if it satisfies the following properties:
\begin{itemize}
\item[(1)] The homotopically coCartesian square
\[
 \xymatrix@1{ I \ar[r]^j \ar[d] & V \ar[d]^{\mathrm{cok}(j)} \\
 0 \ar[r] & \mathrm{Coker}(j)
}
\] is also homotopy Cartesian.
 \item[(2)] Write $\tilde{I} = I \otimes I$. Then $\tilde{I}$ is cofibrant
	    and $\mu_{\tilde{I}}: \tilde{I} \otimes \tilde{I} \to
	   \tilde{I}$ has a homotopy section $\Delta : \tilde{I} \to
	   \tilde{I} \otimes \tilde{I}$ satisfying the associativity:
\[
  \xymatrix@1{  \tilde{I}\ar[d] \ar[r]  & \tilde{I} \otimes \tilde{I}  \ar[r]& ( \tilde{I} \otimes \tilde{I} )  \otimes \tilde{I} \ar[d] \\
  \tilde{I} \otimes \tilde{I}  \ar[rr] & & \tilde{I}  \otimes  ( \tilde{I} \otimes \tilde{I} ). }
\]
\item[(3)] The tensor product $\tilde{I} \otimes \tilde{I}$ is a group like object. That is, the unit morphism $ \tilde{I} \otimes \tilde{I} \to \Omega\Sigma (  \tilde{I} \otimes \tilde{I} )  $ is a weak equivalence, where $\Omega$ is the pointed loop functor  and $ \Sigma$ the suspension functor on the pointed model category $\mathcal{M}$.
\end{itemize}
\end{definition}
\begin{proposition}
Let $\mathcal{M}$ be a symmetric monoidal model category with a monoidal
unit $V$, and $j: I \to V$ a homotopically idempotent commutative Smith
ideal of $V$. Then the product $\mu_{\tilde{I}}: \tilde{I} \otimes
\tilde{I} \to \tilde{I}$ is a weak equivalence.
\end{proposition}
\begin{proof}
 We show that the homotopy kernel of $\mu_{\tilde{I}}: \tilde{I} \otimes
\tilde{I} \to \tilde{I}$ is contractible.  Let $K$ denote the homotopy
kernel of $\mu_{\tilde{I}}$.  One has a diagram
\[
  \xymatrix@1{  \tilde{I} \ar[r]^{\Delta} \ar[d] &   \tilde{I} \otimes \tilde{I} \ar[r]^{\mu_{\tilde{I}}} \ar[d] & \tilde{I} \ar[d] \\
                    0    \ar[r]  &                  K   \ar[r] \ar[d]          & 0  \ar[d] \\
       &   0 \ar[r] &      0 \amalg_{\tilde{I} \otimes \tilde{I}}\tilde{I}, }
\]
where all squares are homotopy Cartesian, and the upper two squares are homotopy Cartesian.  Since weak equivalences are closed under
retracts, $\tilde{I} \to \Omega \Sigma \tilde{I} $ and $K \to \Omega
\Sigma K $ are weak equivalence.  Since $ \mu_{\tilde{I}}$ has a
homotopy section, the homotopy cokernel $0 \amalg_{\tilde{I} \otimes
\tilde{I}}\tilde{I} $ is contractible. Therefore, $\Sigma K$ is
contractible, implying that $K$ is also contractible.
\end{proof}

\begin{remark}
Since the homotopy cokernel $0 \amalg_{\tilde{I} \otimes
\tilde{I}}\tilde{I} $ is contractible, the pushout product
$\mu_{\tilde{I}}:\tilde{I} \amalg_{\tilde{I} \otimes \tilde{I}}
\tilde{I} \to \tilde{I} $ is a weak equivalence, implying that $
\mu_{\tilde{j}}: \tilde{j} \Box \tilde{j} \to \tilde{j} $ is a
projective equivalence.
\end{remark}

\begin{example}
Let $V$ be a unital commutative ring and $\mathfrak{m}$ an
idempotent ideal of $V$. Under the assumption that $\tilde{\mathfrak{m}}
=\mathfrak{m} \otimes_V \mathfrak{m}$ is a flat $V$-module, the product
$\mu_{\mathfrak{m}}:      \mathfrak{m} \otimes_V  \mathfrak{m}   \to \mathfrak{m}$ induces isomorphisms $\mu_{\mathfrak{m}} \otimes \mathrm{id}_{\mathfrak{m}}: \tm \otimes_V  \mathfrak{m} \to  \tm$ and  $ \mathrm{id}_{\mathfrak{m}}\otimes  \mu_{\mathfrak{m}}: \mathfrak{m} \otimes_V \tm \to  \tm$.  Hence the composition
\[
 \Delta= ( \mu_{\mathfrak{m}} \otimes  \mu_{\mathfrak{m}} )^{-1}  =(\mathrm{id}_{\tm} \otimes \mu_{\mathfrak{m}})^{-1} \circ  (\mu_{\mathfrak{m}} \otimes \mathrm{id}_{\mathfrak{m}})^{-1}: \tm \to \tm \otimes_{V} \tm
\]
is an isomorphism, satisfying the coherence condition in Definition~\ref{idem-def}.
\end{example}

\begin{example}[\cite{zbMATH00063621,zbMATH06306453} Torsion and derived completion]
Let $A$ be a commutative ring and $\mathfrak{a} \subset A$ a finitely
generated weakly proregular ideal. Consider the injective model category
$\Ch(A)$ and let $K_\infty(\mathfrak{a})$ denote the infinite dual
Koszul complex. Then, on the derived category $D(A)$, the derived torsion functor $\DR\Gamma_{\mathfrak{a}}(-)$
is equivalent to the tensor functor $K_\infty(\mathfrak{a}) \otimes_A^{\DL} (-)$,
and the derived completion functor $\DL\Lambda_{\mathfrak{a}}(-)$
is equivalent to the internal-hom functor $
 \Map_A(K_\infty(\mathfrak{a}),\,-)$. Note that $K_\infty(\mathfrak{a})$ is homotopy idempotent:
$ K_\infty(\mathfrak{a}) \otimes_A^{\DL} K_\infty(\mathfrak{a})
 \simeq K_\infty(\mathfrak{a})$. Therefore, the adjunction
\[
 K_\infty(\mathfrak{a}) \otimes_A^{\DL} (-) :D(A)
 \rightleftarrows D(A)  :  \Map_A(K_\infty(\mathfrak{a}),\,-)
\]
gives a bilocalization of $D(A)$, where derived torsion is realized as the tensor side and derived completion as the internal-hom side.
In this sense, the derived completion functor is an instance of the
internal-hom side of our general tensor--internal-hom bilocalization.
This viewpoint is compatible with the equivalence between
firm and closed modules over an idempotent ideal: derived
$\mathfrak{a}$-torsion complexes are the fixed objects of
$K_\infty(\mathfrak{a}) \otimes_A^{\DL}(-)$, while derived
$\mathfrak{a}$-complete complexes are the fixed objects of
$\Map_A(K_\infty(\mathfrak{a}),\,-)$. The adjunction restricts to the
Greenlees--May duality between the corresponding full
subcategories~\cite{zbMATH06306453,PSYErratum}.
Hence the commutative Smith ideal $j:I \to A$ obtained from a cofibrant
replacement of $K_\infty(\mathfrak{a}) \to A$ gives a fundamental example of a
homotopically idempotent commutative Smith ideal governing this duality.
\end{example}

\begin{definition}
Let $j: I \to V$ be a homotopically idempotent commutative Smith ideal of the monoidal unit $V$. A morphism $f: M \to N$ in $\mathcal{M}$ is called an {\it almost
weak equivalence} if the induced morphism $\tilde{I} \otimes Q(f):
\tilde{I} \otimes Q(M) \to\tilde{I} \otimes Q(N) $ is a weak
equivalence, and $\mathcal{M}^{\rm al}$ denotes the Bousfield localization of $\mathcal{M}$ with respect to the class of almost weak equivalences.
\end{definition}


\begin{definition}
An object $M$ of $\mathcal{M}$
is {\it  almost local} if, for any almost weak
equivalence $f: N_1 \to N_2$, the induced morphism
\[
f^*:
\Hom_{\mathrm{Ho}\mathcal{M}}(N_2,\,M) \to
\Hom_{\mathrm{Ho}\mathcal{M}}(N_1,\,M)
\]
is an isomorphism.
Dually, $M$ is  {\it almost colocal} if, for any almost weak equivalence $f: N_1 \to N_2$, the induced map
\[
f_*: \Hom_{\mathrm{Ho}\mathcal{M}}(M,\,N_1) \to \Hom_{\mathrm{Ho}\mathcal{M}}(M,\,N_2)
\]
is an isomorphism.
\end{definition}

\begin{definition}
Let $j:I \to V $ be a homotopically idempotent commutative Smith
ideal. An object $E$ is {\it homotopically $I$-closed} if the induced
morphism
\[
  \Map_{\mathcal{M}}(\tilde{j},\,E ):  \Map_{\mathcal{M}}(V,\,E ) \to \Map_{\mathcal{M}}(\tilde{I},\,E )
\]
is a weak equivalence. An object $E$ is {\it homotopically $I$-firm} if $\tilde{j}: \tilde{I} \to V$ induces a weak equivalence
\[
  \tilde{j} \otimes E :   \tilde{I} \otimes E \to E.
\]
\end{definition}

\begin{remark}
In Quillen's note~\cite[Definition 5.2]{Q}, $I$-closed modules are said to be {\it closed}. However, this paper considers closed monoidal model categories, so we use the terminology {\it $I$-closed objects} for a given homotopically idempotent Smith ideal $I$.
\end{remark}

We verify the equivalence between almost local objects and homotopically $I$-closed objects.
\begin{proposition}
\label{closed-local}
Let $\mathcal{M}$ be a pointed symmetric monoidal model category with a monoidal unit $V$ and $j: I \to V$ a homotopically idempotent Smith ideal.
\begin{enumerate}[(1)]
 \item For any object $M$, $\tilde{j}: \tilde{I} \to V$ induces an almost weak equivalence $ \tilde{j}:\tilde{I} \otimes M \to V \otimes M \to M$.
 \item An object $E$ is almost local if and only if it is homotopically $I$-closed.
\end{enumerate}
\end{proposition}
\begin{proof}
To prove (1), consider $Q(\tilde{j}): \tilde{I} \to Q(V) \to V$ as a factorization, where the first morphism is a cofibration and the second is a trivial fibration. Then, by the unit condition of symmetric monoidal model categories, $\tilde{I} \otimes Q(V) \otimes Q(M) \to \tilde{I} \otimes V \otimes Q(M)$ is a weak equivalence by definition. Additionally, since $\tilde{I} \otimes \tilde{I} \to \tilde{I} \otimes Q(V)$ is a weak equivalence of cofibrant objects, and $- \otimes Q(M)$ is a left Quillen functor, by Ken Brown's lemma~\cite[p.6 Lemma 1.1.12]{Hoveybook}, we have that $\tilde{I} \otimes \tilde{I} \otimes Q(M) \to \tilde{I} \otimes Q(V) \otimes Q(M)$ is also a weak equivalence. Consequently, $\tilde{I} \otimes \tilde{I} \otimes Q(M) \to \tilde{I} \otimes Q(M)$ is a weak equivalence.

The if-direction of (2) is clear by the definition of almost weak equivalences. Now, let $E$ be an almost local object. Since the canonical morphism $\tilde{I} \otimes M \to M$ is an almost weak equivalence for any cofibrant object $M \in \mathcal{M}$ and $\tilde{I}$ is cofibrant, we have a chain of isomorphisms: $\Hom_{\mathrm{Ho}(\mathcal{M})}(M, E) \to \Hom_{\mathrm{Ho}(\mathcal{M})}(M, \Map_{\mathcal{M}} ( \tilde{I}, E)) \simeq \Hom_{\mathrm{Ho}(\mathcal{M})}(\tilde{I} \otimes M, E)$. Thus, by the Yoneda lemma, the induced morphism $E \to \Map (\tilde{I}, E)$ is a weak equivalence, implying the only-if direction of (2).
%
\qed
\end{proof}

\begin{theorem}
\label{theoremA}
Let $\mathcal{M}$ be a pointed symmetric monoidal model category with a monoidal unit object $V$, let $j: I \to V$ be a homotopically idempotent commutative Smith ideal, and let $\mathcal{M}^{\rm al}$ denote the Bousfield localization with respect to the collection of almost weak equivalences.  Then the internal-hom functor
\[
 \Map_{\mathcal{M}}(\tilde{I},\, -):  \mathcal{M}      \to    \mathcal{M}^{\rm al}
\]
is equivalent to the Bousfield bilocalization functor.

Furthermore, let $\mathcal{M}^{\rm fi}$ denote the homotopically essential image of $\tilde{I} \otimes Q: \mathcal{M} \to \mathcal{M}$. Then the induced adjunction
\[
  \tilde{I} \otimes Q:  \mathcal{M}^{\rm al}    \to
\mathcal{M}^{\rm fi}:\Map_{\mathcal{M}}(\tilde{I},\, -)
\]
is a Quillen equivalence. In particular, for any object $E$, the canonical morphism
\[
  E \to \Map_\mathcal{M}(\tilde{I},\,E)
\]
is an almost weak equivalence.
\end{theorem}
\begin{proof}

%
%
Let $L^{\rm al} : \mathcal{M}\to \mathcal{M}^{\rm al} $ be another Bousfield
localization by almost weak equivalences. By the universality of $L^{\rm al} $, the natural transformation $\mathrm{Id}
\to \Map_{\mathcal{M}}( \tilde{I} ,\, -)$ homotopically factors as
\[
 \mathrm{Id} \overset{\varphi}{\to} L^{\rm al} (-) \overset{\psi}{\to}
 \Map_{\mathcal{M}}( \tilde{I} ,\, -).
\]
By (2) in Proposition~\ref{closed-local}, the composition $ \psi(L^{\rm al} ) \circ
\varphi(L^{\rm al} ): L^{\rm al}  \to L^{\rm al}  \circ L^{\rm al}  \to \Map_{\mathcal{M}}( \tilde{I}
,\, L^{\rm al} (-)) $ is an equivalence. Therefore, the induced transformation
\[
\Map_{\mathcal{M}}(\tilde{I} ,\, \varphi): \Map_{\mathcal{M}}(
\tilde{I} ,\, -) \to \Map_{\mathcal{M}}( \tilde{I} ,\, L^{\rm al} (-)) 
\]
is
identified with a homotopy section of $\psi: L^{\rm al}  \to
\Map_{\mathcal{M}}(\tilde{I} ,\, -) $. Further, since the functor $
\Map_{\mathcal{M}}(\tilde{I} ,\, -)$ is homotopically idempotent, the
factorization
\[
   \Map_{\mathcal{M}}(\tilde{I} ,\, \psi) \circ  \Map_{\mathcal{M}}(\tilde{I} ,\, \varphi): \Map_{\mathcal{M}}(\tilde{I} ,\, -) \to \Map_{\mathcal{M}}(\tilde{I} ,\, L^{\rm al} (-)) \to \Map_{\mathcal{M}}(\tilde{I} ,\, -)\circ \Map_{\mathcal{M}}(\tilde{I} ,\, -)
\]
implies that $\Map_{\mathcal{M}}( \tilde{I} ,\,-)$ is a homotopy retraction of $L^{\rm al} $.  Hence, the natural transformation $\psi: L^{\rm al}  \to \Map_{\mathcal{M}}(\tilde{I},\,-)$ is a weak equivalence.

We show that the functor $\tilde{I} \otimes Q: \mathcal{M} \to
\mathcal{M}$ induced a left Quillen equivalence $\tilde{I} \otimes Q:
\mathcal{M}^{\rm al} \to \mathcal{M}^{\rm fi}$.
Since the functor
$\tilde{I} \otimes Q: \mathcal{M} \to \mathcal{M}$ sends almost weak
equivalences to weak equivalences, it factors through $\mathcal{M}^{\rm
al}$. By (1) in Proposition~\ref{closed-local} and the equivalence
$\psi: L^{\rm al} \to \Map_{\mathcal{M}}(\tilde{I},\, -)$, the unit
$\mathrm{Id} \to \Map_{\mathcal{M}}(\tilde{I},\, -) \circ ( \tilde{I}
\otimes Q ) $ induces an equivalence on the homotopy category of
$\mathcal{M}^{\rm al}$, implying that it is a left Quillen equivalence.
\qed
\end{proof}

\section{Commutative monoid objects in almost mathematics of pointed closed symmetric monoidal model categories}
\label{sec:almost-alg}
First, we consider the module category over a unital ring.  Let $V$ be a commutative unital ring with an idempotent ideal $\mathfrak{m}$. Assume that $\tm$ is a flat $V$-module.  In the language of Smith ideal theory, a $V$-homomorphism $\tilde{j}:\tm \to V$ is regarded as a commutative Smith ideal of $V$. For any $V$-algebra $A$, the almost unitalization $A_{!!}$ can be represented by the pushout product
\[
\tilde{j}    \Box (V \to A) = (V  \amalg_{\tm } ({\tm \otimes A} )      ) \to A,
\]
which is an almost isomorphism. In Gabber--Ramero's text~\cite{GR}, $A_{!!}$ is defined as the source of $\tilde{j} \Box (V \to A)$ without using Smith ideal theory. From the viewpoint of Smith ideal theory, one can easily check that the induced ring homomorphism $A_{!!} \to A$ is an almost isomorphism.  Indeed, by the associativity of pushout products,
\[
 (\tm \otimes_{V} {A}_{!!} \to \tm \otimes_{V} A  ) \simeq    L_1(\tm)    \Box (\tilde{j} \Box (V \to A)) \\
\to ( L_1(\tm) \Box \tilde{j}  )\Box (V \to A) \simeq  \mathrm{id}_{\tm } \Box ( V \to A)\simeq  \mathrm{id}_{\tm \otimes_V A }
\]
is an isomorphism between arrows, implying that those sources $\tm \otimes_V A_{!!}$ and $\tm \otimes_V A$ are isomorphic. We generalize those properties to model categories.

\subsection{Almost unitalization functors of commutative algebra objects of symmetric monoidal model categories}

In general the adjunction $\mathrm{Sym}: \mathcal{M} \rightleftarrows
\CAlg(\mathcal{M}):U$ does not induces a model structure on
$\CAlg(\mathcal{M})$ but it has the semi-model structure defined in
\cite[Definition 2.3]{spitzweck2001operads} or \cite[Definition
3.9]{WH2017}. A semi-model structure only requires that trivial cofibrations
with cofibrant domains have the left lifting property with respect to
all fibrations, and only maps with cofibrant domains factor into a
trivial cofibration followed by a fibration. If all objects of a
semi-model category are cofibrant, it is just a
model category (see \cite[Definition 3.9]{WH2017}).  Left and right Quillen functors between semi-model categories are defined in the same way as those of model categories.

For any class of morphisms $\mathcal{I}$, let
$\mathrm{Cell}(\mathcal{I})$ denote the class of morphisms that are
transfinite compositions of pushouts of elements of $\mathcal{I}$, and
a morphism in $\mathrm{Cell}(\mathcal{I}) $ is said to be a {\it relative
$\mathcal{I}$-cell complex}.

Let $\mathcal{M}$ be cofibrantly generated with a set of generating
cofibrations $\mathcal{I}$ and a set of generating trivial cofibrations
$\mathcal{J}$.  By White~\cite[Corollary 3.8]{WH2017}, if the domains of
$\mathcal{I}$ (resp. $\mathcal{J}$) are small relative to $\mathrm{Cell}
(\mathcal{I} \otimes \mathcal{M})$ (resp. $\mathrm{Cell} ( \mathcal{J}
\otimes \mathcal{M})$), then the category of commutative $R$-algebras
has a semi-model structure inherited from $\mathcal{M}$ for any
commutative monoid object $R$.


We recall the definition of lax monoidal Quillen adjunctions:
\begin{definition}[\cite{zbMATH01924523} Definition 3.6]
\label{lax-Quillen}
Let $ F : \mathcal{M } \rightleftarrows \mathcal{N} : U $ be a Quillen
adjunction of monoidal model categories. Then we say that $(F,\,U)$ is
{\it lax Quillen monoidal} if it satisfies the following properties:
\begin{itemize}
 \item[(1)] The right Quillen functor $U$ is lax monoidal.

 \item[(2)] For any cofibrant objects $M$ and $N$ in $\mathcal{M}$, the
	    induced morphism
\[
\nabla_{M,\,N}: F(M \otimes N) \to F(M) \otimes F(N)
\]
by the canonical morphism $\Delta_{M,\,N}: M \otimes N \to U(F(M)) \otimes
	    U(F(N)) \to U(L(M) \otimes L(N) )$ is a weak equivalence.
\item[(3)] For any cofibrant replacement $q(\mathbf{1}_\mathcal{M} )  : Q(\mathbf{1}_\mathcal{M} ) \to	    \mathbf{1}_\mathcal{M} $ of the monoidal units
	    $\mathbf{1}_\mathcal{M}$ of $\mathcal{M}$ and $\mathbf{1}_\mathcal{N}$ of $\mathcal{N}$, the composition
\[
   F( Q(\mathbf{1}_\mathcal{M} ) ) \to  F( \mathbf{1}_\mathcal{M}  ) \to
\mathbf{1}_\mathcal{N}
 \]
is a weak equivalence.
\end{itemize}
 \end{definition}

\begin{proposition}
\label{Quillen-monoidal} Let $j: I \to M$ be a homotopically idempotent
commutative Smith ideal. Then the Quillen equivalence
\[
 \tilde{I} \otimes Q:  \mathcal{M}^{\rm al}  \rightleftarrows   \mathcal{M}^{\rm fi}: L^{\rm al}
\]
is lax Quillen monoidal.
\end{proposition}
\begin{proof}
Since the class of almost weak equivalences between cofibrant objects is
closed under tensor product, the localization $L^{\rm al} :\mathcal{M}
\to \mathcal{M}^{ \rm al} $ is lax monoidal. Furthermore, given any pair
of cofibrant objects $M$ and $N$, the product $ \mu_{\tilde{I}}:
\tilde{I} \otimes \tilde{I} \to \tilde{I}$ induces a weak equivalence
\[
  (\tilde{I}\otimes Q(M)) \otimes
  (\tilde{I}\otimes Q(N)) \simeq  (\tilde{I}\otimes M) \otimes
  (\tilde{I}\otimes N)  \to \tilde{I} \otimes (M \otimes N).
\]
Furthermore, since $Q(M) \otimes Q(N)$ is cofibrant, $Q(M) \otimes Q(N) \to M \otimes N$ factors through the cofibrant replacement $Q(M \otimes N)$, implying that  Condition (2) of Definition~\ref{lax-Quillen} holds. Condition (3) is obvious.  \qed
\end{proof}

\begin{corollary}
\label{closed-alg} Let $j: I \to M$ be a homotopically idempotent
commutative Smith ideal. The mapping functor
$\Map_{\mathcal{M}}(\tilde{I},\,-): \mathcal{M} \to \mathcal{M}^{\rm
al}$ is lax monoidal, and the Quillen equivalence $( \tilde{I} \otimes
Q,\, \Map_{\mathcal{M}} (\tilde{I},\,R(-)))$ induces on their
commutative monoid objects
\[
 \tilde{I} \otimes Q: \CAlg( \mathcal{M}^{\rm al})  \rightleftarrows   \CAlg(\mathcal{M}^{\rm fi}) : \Map_\mathcal{M}(\tilde{I},\,-)
\]
between semi-model categories.
 \qed
\end{corollary}

Now, for any commutative monoid object $A$ of $\mathcal{M}$, we define a monoid object $A_{!!}$ and the induced morphism $A_{!!} \to A$ is an almost weak equivalence under some mild conditions.
\begin{proposition}
\label{almost-unitalization} Let $\mathcal{M}$ be a pointed closed
symmetric monoidal model category with a monoidal unit object $V$ and
$A$ a commutative algebra object of $\mathcal{M}$.  Let $j : I\to V $ be
a homotopically idempotent commutative Smith ideal of $V$. For any
monoid morphism $f: A \to B$ of commutative monoid objects, the induced
morphism
\[
   \tilde{j}\Box f:  (V \otimes A) \amalg_{\tilde{I} \otimes A} (\tilde{I} \otimes B) \to B
\]
is a monoid morphism.
\end{proposition}
\begin{proof}
Replacing $V$ and $A$ and considering $\CAlg(\mathcal{M})_{A/}$, it is
   sufficient to prove that \[
\tilde{j}\Box L_1(A): A_{!!} = V
   \amalg_{\tilde{I} } (\tilde{I} \otimes A) \to A
\]
is a monoid morphism between commutative monoid objects. The monoid
structure on $A_{!!}$ is induced by those morphisms
\[
  \eta_A: V \otimes A \to A \ \text{ and } \  \tilde{j} \otimes \mathrm{Id_A}:  \tilde{I} \otimes A \to V \otimes A \to A,
\]
and the induced morphism
\[
  \mu_{\tilde{I}} \otimes \mu_{A} : ( \tilde{I} \otimes A ) \otimes  ( \tilde{I} \otimes A )\to  ( \tilde{I} \otimes A ).
\]
Then those morphisms induces $\eta_{A_{!!}}: V \otimes A_{!!}  \to A_{!!} $
and $\mu_{!!} : A_{!!} \otimes A_{!!} \to A_{!!} $ satisfy the
associativity, the unitality, and the commutativity inherited by all of $V$, $\tilde{I}$, and $\tilde{I} \otimes {A}$. Hence, $A_{!!}$ has a commutative monoid structure, and the square $(\tilde{I}  \to V)  \to (\tilde{I} \otimes A \to A) $ is a morphism of commutative Smith ideals, implying that $ A_{!!} \to A$ is a morphism of commutative monoid objects. \qed
\end{proof}
\begin{proposition}
\label{firm-alg} Let $\mathcal{M}$ be a pointed closed symmetric
monoidal model category with a monoidal unit object $V$ and $A$ a
commutative algebra object of $\mathcal{M}$. Let $j: I\to V $ be a
homotopically idempotent commutative Smith ideal of $V$.  If $0 \to V
\to A$ or $V \to A$ is a cofibration, then the pushout product
\[
 (\tilde{j}\Box L_1(A):  A_{!!} \to A ) \to (\mathrm{Id}_{A})
\]
is an almost weak equivalence, where we write
 \[
   A_{!!} =\mathrm{Ev}_0( (\tilde{j}:\tilde{I} \to V    ) \Box (L_1(A):V \to A  )) = V \amalg_{\tilde{I}} (\tilde{I} \otimes A)).
\]
\end{proposition}
\begin{proof}
Assume that $0 \to A$ is a cofibration. By definition of homotopy
pushout, let $ \tilde{I} \to (\tilde{I} \otimes A)' \to \tilde{I}
\otimes A$ be a factorization of a cofibration to a trivial fibration, then
the induced morphisms
\[
    Q(V) \amalg_{\tilde{I}} (\tilde{I} \otimes A)'  \to Q(A_{!!}) \to A_{!!}
\]
are weak equivalences. Therefore $ \tilde{I} \otimes (\tilde{I} \otimes
A)' \to \tilde{I} \otimes Q(A_{!!})$ is a weak equivalence. Since
$(\tilde{I} \otimes A)' \to \tilde{I} \otimes A$ is a weak equivalence
between cofibrant objects, $ \tilde{I} \otimes (\tilde{I} \otimes A)'
\to \tilde{I} \otimes ( \tilde{I} \otimes A) \to \tilde{I} \otimes A$ is
a composition of weak equivalences.  Therefore $ \tilde{I} \otimes
Q(A_{!!}) \to \tilde{I} \otimes A$ is a weak equivalence.

Next, we assume that $V \to A$ is a cofibration. Then $\tilde{I} \otimes
A$ is cofibrant. Therefore $\tilde{I} \otimes A \to A$ factors as
\[
 \tilde{I} \otimes A \to Q(A_{!!})  \to Q(A) \to A,
\]
inducing a tower of  morphisms
\[
  \cdots \to   \tilde{I}^{\otimes 3} \otimes Q(A) \to  \tilde{I}^{\otimes 3} \otimes A \to   \tilde{I}^{\otimes 2} \otimes   Q(A_{!!})  \to     \tilde{I}^{\otimes 2} \otimes Q(A) \to \tilde{I} \otimes A \to  \tilde{I} \otimes Q(A_{!!})  \to  \tilde{I} \otimes Q(A).
\]
Since $\tilde{I}^{\otimes 2} \otimes Q(A_{!!})  \to \tilde{I}^{\otimes
2} \otimes Q(A) \to \tilde{I}^{\otimes 2} \otimes A \to \tilde{I}
\otimes Q(A_{!!})$ and $ \tilde{I}^{\otimes 2} \otimes Q(A) \to
\tilde{I} \otimes A \to \tilde{I} \otimes Q(A_{!!})  \to \tilde{I}
\otimes Q(A)$ are weak equivalences, 
we obtain that $ \tilde{I} \otimes Q(A_{!!})  \to \tilde{I} \otimes Q(A) $ is also a
weak equivalence.  \qed
\end{proof}

We verify that, for any commutative monoid object $A$, the internal-hom object $\Map_{\mathcal{M}} (\tilde{I},\,A)$ has a commutative monoid
structure in $\mathcal{M}$ inherited from that of $A$ and the homotopy section
$\Delta: \tilde{I} \to \tilde{I} \otimes \tilde{I}$ satisfying the
coherent conditions in Definition~\ref{idem-def}. Let
$\mathrm{ev}_{\tilde{I}}: \tilde{I} \otimes
\Map_{\mathcal{M}}(\tilde{I},\,A ) \to A$ denote the evaluation morphism
induced by the identity on $\Map_{\mathcal{M}}(\tilde{I},\,A )$. Then
the composition
\[
   \Map_{\mathcal{M}}(\tilde{I},\,A  ) \otimes  \Map_{\mathcal{M}}(\tilde{I},\,A  ) \otimes  (\tilde{I} \otimes \tilde{I})
 \cong   (  \tilde{I}  \otimes \Map_{\mathcal{M}}(\tilde{I},\,A  )  )  \otimes   (  \tilde{I}  \otimes \Map_{\mathcal{M}}(\tilde{I},\,A  )  )     \overset{\mathrm{ev}_{\tilde{I}} \otimes \mathrm{ev}_{\tilde{I}}}{\to}   A \otimes A \overset{\mu_A}{\to} A
\]
determines a canonical monoid structure
\[
 \mu_{ \Map_{\mathcal{M}}(\tilde{I},\,A  )  }: \Map_{\mathcal{M}}(\tilde{I},\,A  ) \otimes  \Map_{\mathcal{M}}(\tilde{I},\,A  ) \to    \Map_{\mathcal{M}}(\tilde{I} \otimes \tilde{I},\,A  )    \overset{\Delta^{*}}{\to}    \Map_{\mathcal{M}}(\tilde{I},\,A  ),
\]
satisfying associativity and unitality by Condition (3) in Definition~\ref{idem-def}.

In the case that the monoidal structure on $\mathcal{M}$ is closed, both symmetric monoidal structures on the arrow category $\mathrm{Ar}(\mathcal{M})$
are closed by Hovey~\cite[Theorem 1.2]{Smith-ideals}. For any pair of
morphisms $(f:X_0 \to X_1,\,g:Y_0 \to Y_1)$ in $\mathcal{M}$, the
internal-hom objects are determined as follows: The internal-hom object
$\Map_{\mathcal{M}}^\otimes(f,\,g)$ of the diagonal product monoidal
structure is defined by
\[
  \Map^{\otimes}(f,\,g):\Map_{\mathcal{M}}(X_0,\,Y_0) \times_{\Map_\mathcal{M}(X_0,\,Y_1)} \Map_\mathcal{M}(X_1,\,Y_1) \to  \Map_\mathcal{M}(X_1,\,Y_1)
\] and $\Map^\Box(f,\,g)$ of the pushout monoidal structure is defined by
\[
  \Map^\Box(f,\,g): \Map_\mathcal{M}(X_1,\,Y_0) \to \Map_\mathcal{M}(X_0,\,Y_0) \times_{\Map_\mathcal{M}(X_0,\,Y_1)} \Map_\mathcal{M}(X_1,\,Y_1).
\]

This section's main result is the following.
\begin{theorem}
\label{adjoint} Let $\mathcal{M}$ be a pointed closed symmetric monoidal
model category, and $V$ a monoidal unit of $\mathcal{M}$ with a homotopically
idempotent commutative Smith ideal $j:I \to V$. Assume that $\mathcal{M}$ is left proper and combinatorial. Then the Quillen adjunction
\[
 \tilde{j}\Box (-): \mathrm{Ar}(\CAlg(\mathcal{M})  ) \rightleftarrows \mathrm{Ar}(\CAlg(\mathcal{M})  ) : \Map^\Box(\tilde{j},\,-)
\]
 induces a Quillen equivalence
\[
\tilde{j}\Box (-): \mathrm{Ar}(\CAlg(\mathcal{M}^{\rm al})  ) \rightleftarrows \mathrm{Ar}(\CAlg(\mathcal{M}^{\rm al})  ): \Map^\Box(\tilde{j},\,-),
\]
which is equivalent to the trivial Quillen equivalence
\[
L_0 \circ \mathrm{Ev}_1  : \mathrm{Ar}(\CAlg(\mathcal{M}^{\rm al})  ) \rightleftarrows \mathrm{Ar}(\CAlg(\mathcal{M}^{\rm al})  ):  U_1 \circ    \mathrm{Ev}_0.
\]
\end{theorem}
\begin{proof}
Since the subcategory $\CAlg(\mathcal{M})$ is closed under homotopy
limits, for any morphism $f: A \to B$ of commutative monoid objects,
$\Map^\Box(\tilde{j},\,f):A \to \Map_{\mathcal{M}}(\tilde{I},\,A)
\times_{\Map_{\mathcal{M}}(\tilde{I},\,B)} B$ is a morphism of monoid
objects by Corollary~\ref{closed-alg}. The functor $
\Map^\Box(\tilde{j},- ):\mathrm{Ar}(\mathcal{M}) \to
\mathrm{Ar}(\mathcal{M}) $ can be restricted on
$\mathrm{Ar}(\CAlg(\mathcal{M}))$.

By Proposition~\ref{firm-alg}, $ j \Box \mathrm{Id} \to \mathrm{id}_{Ev_{1}(-)}
=L_0(\mathrm{Ev}_1((-)))$ is a projective equivalence in $
\CAlg(\mathcal{M}^{\rm al})$. Therefore, their Quillen right adjoints
$\Map^\Box (\tilde{j},\,-)$ and $U_1 \circ \mathrm{Ev}_0$ are weakly
equivalent on $\mathrm{Ar}(\CAlg(\mathcal{M}^{\rm al}))$. \qed
\end{proof}


We obtain an analogue of Gabber--Ramero~\cite[p.22, Proposition 2.2.29]{GR} as a corollary of Theorem~\ref{adjoint}.
\begin{corollary}
\label{app-almost} Let $\mathcal{M}$ be a left proper and combinatorial
pointed closed symmetric monoidal model category with a monoidal unit
object $V$ and a homotopically idempotent commutative Smith ideal $j: I
\to V$.  Then the functor $(-)_{!!} =\mathrm{Ev}_0( \tilde{j} \Box
L_1(-) ) : \CAlg (\mathcal{M}) \to \CAlg(\mathcal{M}) $ is a left
Quillen functor, whose right adjoint is the composition
\[
 \mathrm{Ev}_1( \Map^\Box( \tilde{j},\,U_0(-))) : \CAlg(\mathcal{M}) \to \CAlg(\mathcal{M})
\]
Furthermore, the left Quillen functor induces
a left Quillen equivalence \[
 (-)_{!!}:   \CAlg(\mathcal{M}^{\rm al})  \to    \CAlg(\mathcal{M})_{!!},
\]
where $\CAlg(\mathcal{M})_{!!}$ denotes the homotopically essential image of the
endofunctor
\[
	\mathrm{Ev}_0( \tilde{j} \Box L_1(- )) :
\CAlg (\mathcal{M}) \to \CAlg(\mathcal{M}).
\]
\end{corollary}
\begin{proof}
Since all of $L_1$, $\tilde{j}\Box (-)$, and $\mathrm{Ev}_0$ are left
Quillen functors, the composition $(-)_{!!}$ is also a left Quillen
functor, and its right adjoint is the composition of the corresponding right adjoint
functors
\[
	\mathrm{Ev}_1 \circ  \Map^\Box( \tilde{j},\,-) \circ U_0
: \CAlg(\mathcal{M}) \to \CAlg(\mathcal{M}).
\]
By Theorem~\ref{adjoint}, one has a weak equivalence $L^{\rm al} \circ
 (-)_{!!} \circ L^{\rm al} \simeq L^{\rm al} $, and by
 Proposition~\ref{firm-alg}, the endofunctor $(-)_{!!}: \CAlg
 (\mathcal{M}) \to \CAlg(\mathcal{M}) $ is weakly idempotent. Hence,
 those compositions \[ L^{\rm al} \circ(-)_{!!}: \CAlg(\mathcal{M}^{\rm
 al}) \to \CAlg(\mathcal{M})_{!!} \to \CAlg(\mathcal{M}^{\rm al})
\]
and
\[
  (-)_{!!} \circ  L^{\rm al}  :\CAlg(\mathcal{M})_{!!} \to    \CAlg(\mathcal{M}^{\rm al}) \to \CAlg(\mathcal{M})_{!!}
\]
are equivalent to the identity functors, respectively. \qed
\end{proof}

\section{Bousfield bilocalization of stable model categories}
\label{sec:bilocal}
A pointed model category is {\it stable} if its
homotopy category is a triangulated category. In a stable model
category, a homotopy Cartesian square is also homotopy coCartesian and
vice versa. Further, in this section, we consider stable monoidal model
categories generated by their monoidal unit objects. Such categories are
called {\it monogenic}, as explained in Definition~\ref{w-gen}. In this
section, we prove that Bousfield bilocalization functors determine
homotopically idempotent Smith ideals under the condition that the model
categories are stable, closed symmetric monoidal, and monogenic.
Furthermore, we show that homotopically closed objects with respect to
the Smith ideals determined by Bousfield bilocalization homotopically
coincide with local objects.

\subsection{Monogenic model categories}

Following MacLane~\cite[Chapter V, Section 7, p.127]{zbMATH00195199}, we
recall the definition of generators of categories: Let $\mathcal{M}$ be
a category and $\mathcal{G}$ a class of small objects of
$\mathcal{M}$. The set $\mathcal{G}$ is a {\it generator} of
$\mathcal{M}$ if for any pair of morphisms $h,\,h': X \to Y$ in
$\mathcal{M}$, $h \neq h'$ implies that there exists $S \in \mathcal{G}$
and a morphism $f: S \to X$ such that $h \circ f \neq h' \circ f $.  Following
Hovey's book~\cite[p.183, Definition 7.2.1]{Hoveybook}, we define
{\it weak generators} of model categories.
\begin{definition}[\rm{c.f.} Hovey~\cite{Hoveybook}]
\label{w-gen} Let $\mathcal{M}$ be a model category with a class
$\mathcal{G}$ of objects of $\mathcal{M}$. We say that $\mathcal{G}$
weakly generates $\mathcal{M}$ if the equivalence class of $\mathcal{G}$
generates the homotopy category of $\mathcal{M}$. A symmetric monoidal
model category weakly generated by the small monoidal unit is said to be
{\it monogenic}.
\end{definition}
By the following proposition, for stable model categories, the above definition coincides with \cite[p.183, Definition 7.2.1]{Hoveybook}.
\begin{proposition}
\label{lemGen} Let $\mathcal{M}$ be a pointed closed symmetric monoidal
model category with a class $\mathcal{G}$ of objects of
$\mathcal{M}$. Assume that $\mathcal{G}$ generates $\mathcal{M}$. Then
an object $M$ is contractible if and only if the internal
hom-object $\Map_\mathcal{M}(S,\,M)$ is contractible for any $S \in \mathcal{G}$.
\end{proposition}
\begin{proof}
The if-direction is clear. Let $M$ be an object of $\mathcal{M}$ such
that the internal-hom object $\Map_\mathcal{M}(S,\,M)$ is contractible
for any $S \in \mathcal{G}$.  Since $\mathcal{G}$ generates
$\mathcal{M}$, for any object $N$ and morphism $f:N \to M$, the
condition that $f$ is not null-homotopic implies that there is an
object $S \in \mathcal{G}$ and a morphism $i:S \to N$ such that $f \circ i : S
\to M$ is not null-homotopic, contradicting the generating condition. \qed
\end{proof}

\subsection{Symmetric monoidal bilocalization on stable model categories}
\begin{definition}
Let $\mathcal{M}$ be a model category. A functor $F: \mathcal{M} \to
\mathcal{N}$ is said to be a {\it Bousfield bilocalization} if it is a
Bousfield localization admitting fully faithful left adjoint and fully
faithful right adjoint such that all of those functors are Quillen
functors.
\end{definition}

We consider a Bousfield bilocalization $F: \mathcal{M} \to
L(\mathcal{M})$ and let $F_*$ denote its fully faithful left adjoint and
$U$ its fully faithful right adjoint, respectively.
\[
 \xymatrix@1{
      F: \mathcal{M} \ar[r] & \ar[l]<1mm>^{U}
\ar[l]<-1mm>_{F_*} L( \mathcal{M}).
}
\]
Further, we assume that the Quillen adjunction
\[
F_*: L(\mathcal{M}) \rightleftarrows  \mathcal{M} :F
\]
is lax Quillen monoidal. In particular, for any $M$ and $N$ of
$\mathcal{M}$, the induced morphism
\[
	\nabla: F_*(M) \otimes F_*(N) \to F_*(M \otimes N)
									      \]
is a weak equivalence.
Then the counit $c:F_* \circ F \to \mathrm{Id}_\mathcal{M}$ induces a
commutative Smith ideal $c_V: F_*(F(Q(V))) \to V$. Let $j:I \to V$
denote a cofibrant replacement of the homotopy image of $c_V$. We will
prove that $I$ is homotopically idempotent.

\begin{lemma}
\label{unit} Let $F: \mathcal{M} \to L(\mathcal{M})$ be a Bousfield
bilocalization of stable symmetric monoidal model category, and $F_*$
denote its fully faithful left adjoint. Assume that the Quillen
adjunction
\[
 F_* :  L(\mathcal{M}) \rightleftarrows \mathcal{M} :F
\]
is lax Quillen monoidal. Let $j: I \to V$ be a homotopy image object of
$q(V) \circ c_{Q(V)} : F_*(F(Q(V))) \to Q(V) \to V$. Then all of the
canonical morphisms in the diagram
\[
  I \otimes I \rightarrow F_*(F(Q(V))) \otimes I \rightarrow F_*(F(Q(V))) \otimes F_*(F(Q(V))) \rightarrow F_*(F(Q(V))) \leftarrow I
\]
are weak equivalences.
\end{lemma}
 \begin{proof}
Let $q(V): Q(V) \to V$ be a cofibrant replacement of $V$. Since the
Bousfield localization $F$ is both a left and a right Quillen functor, $F(Q(V)) \to
F(V)$ is a trivial fibration, and $F(Q(V))$ is a cofibrant object of $L(\mathcal{M})$.
Furthermore, since $F_*$ is lax Quillen monoidal, $F(Q(V))$ is
equivalent to a monoidal unit of $L(\mathcal{M})$. Therefore, all
morphisms in
\[
 F(Q(V)) \otimes F(Q(V)) \to F(V) \otimes F(Q(V)) \to F(Q(V)) \to F(V)
\]
are weak equivalences by Ken Brown's lemma~\cite[p.6 Lemma 1.1.12]{Hoveybook}. In particular,
\[
  \mu: F(Q(V)) \otimes F(Q(V)) \to F(Q(V))
\]
is a weak equivalence between cofibrant objects. Since $F_*$ is a left
Quillen functor, by Ken Brown's lemma~\cite[p.6 Lemma 1.1.12]{Hoveybook} again,
\[
  F_*(\mu) : F_* ( F(Q(V)) \otimes F(Q(V))  ) \to  F_* (F(Q(V)))
\]
is a weak equivalence. Note that  the product
\[
\mu_{F_*(F(Q(V)))}: F_*(F(Q(V))) \otimes F_* (F(Q(V))) \to F_*(F(Q(V)))
\]
is homotopic to the pushout of $\nabla: F_*(F(Q(V))) \otimes F_*(F(Q(V))) \to F_*(F(Q(V)) \otimes F(Q(V)))$ along the morphism
$F_*(F(\mu)): F_*( F(Q(V)) \otimes F(Q(V))) \to F_*(F(Q(V)))$.  By the
left properness of $\mathcal{M}$, $\mu: F_*(F(Q(V))) \otimes
F_* (F(Q(V))) \to F_*(F(Q(V)))$ is a weak equivalence.

By the stability of $\mathcal{M}$, the morphism $(j: I \to V )
\to (q(V) \circ c_{Q(V)}: F_* ( F(Q(V)) ) \to V ) $ is a weak
equivalence in the projective model structure. Since $ I \to F_* (
F(Q(V)) ) $ is a weak equivalence of cofibrant objects, by Ken Brown's
lemma again, both morphisms in $I \otimes I \rightarrow F_*(F(Q(V)))
\otimes I \rightarrow F_*(F(Q(V))) \otimes F_*(F(Q(V)))$ are weak
equivalences.
 \qed
\end{proof}

\begin{proposition}
\label{theorem-idem} Let $F: \mathcal{M} \to L(\mathcal{M})$ be a
Bousfield bilocalization of a stable symmetric monoidal model category,
and $F_*$ denote its fully faithful left adjoint. Assume that the
Quillen adjunction
\[
 F_* :  L(\mathcal{M}) \rightleftarrows \mathcal{M} :F
\]
is lax Quillen monoidal.  Then the counit $c:F_* \circ F \to
\mathrm{Id}_\mathcal{M}$ determines a homotopically idempotent
commutative Smith ideal $j:I \to V$ of a monoidal unit object $V$, where
$I$ is a cofibrant replacement of the homotopy image of $q(V) \circ
c_{Q(V)} : F_*(F(Q(V))) \to Q(V) \to V$.
\end{proposition}
\begin{proof}
Since $\mathcal{M}$ is stable, the unit morphism $c_V \to j $ is a
projective weak equivalence in the arrow category of
$\mathrm{Ar}(\mathcal{M})$.
Condition (1) of Definition~\ref{idem-def} is already satisfied.  By
Lemma~\ref{unit}, condition (2) in Definition~\ref{idem-def} holds. By
the stability of $\mathcal{M}$, the right exact functor $I \otimes (-)$
is also left exact, implying that the induced morphism $ \tilde{I}
\otimes I \to \tilde{I} $ is a weak equivalence.  \qed
\end{proof}

\begin{proposition}
\label{firm} Let $\mathcal{M}$ be a pointed closed symmetric monoidal
model category with a monoidal unit $V$ and $F: \mathcal{M}
\to L(\mathcal{M})$ a symmetric monoidal Bousfield bilocalization, and
$F_*: L(\mathcal{M}) \to \mathcal{M} $ denote the homotopically fully
faithful left adjoint. Assume that the Quillen adjunction
\[
 F_* : L(\mathcal{M}) \rightleftarrows  \mathcal{M} :F
\]
is lax Quillen monoidal. Let $j: I \to V$ be a homotopically idempotent
commutative Smith ideal determined by the counit $c: F_* \circ F \to
\mathrm{Id}$. Then
\[
 F(I \otimes Q(M)) \to F(M)
\]
is a weak equivalence for any object $M$.
\end{proposition}
\begin{proof}
Let $j:I \to F_*(F(V))$ be a cofibrant replacement of the homotopy image
of the counit $c_V:F_*(F(V)) \to V $. Then $F(j):F(I) \to \mathrm{Im}(c_V) $ is a
trivial fibration.  Therefore, $F(I) \to F(F_*(F(V))$ is a weak
equivalence.  Since $F_*$ is fully faithful, the natural transformation
$F(c): F \circ F_* \circ F \to F $ is a weak equivalence. Therefore,
$F(I) \to F(F(V))$ is also a weak equivalence.  Therefore
\begin{multline*}
F(\tilde{I}) \to F(I) \otimes F(I) \to F(F_*(F(Q(V)))) \otimes  F(F_*(F(Q(V))))  \\
\to F(Q(V)) \otimes F(Q(V)) \to F(V) \otimes F(Q(V))
 \to F(V) \otimes F(V) \to F(V)
\end{multline*}
is a weak equivalence and $F(\tilde{I} \otimes Q(M)) \to F(\tilde{I})
\otimes F(Q(M)) \to F(V) \otimes F(Q(M)) \to F(Q(M)) $ is a weak
equivalence for any object $M \in \mathcal{M}$. Since $F$ is a right
Quillen functor, for any cofibrant replacement $Q(M) \to M$, $ F(Q(M))
\to F(M)$ is a trivial fibration, giving us the conclusion.  \qed
\end{proof}

\begin{theorem}
\label{theoremB} Let $\mathcal{M}$ be a stable closed symmetric monoidal
model category and $F: \mathcal{M}
\to L(\mathcal{M})$ a Bousfield bilocalization, and
$F_*: L(\mathcal{M}) \to \mathcal{M} $ denote the homotopically fully
faithful left adjoint. Assume that $\mathcal{M}$ is monogenic and the Quillen adjunction
\[
 F_* : L(\mathcal{M}) \rightleftarrows  \mathcal{M} :F
\]
is lax Quillen monoidal. Let $j: I \to V$ be a homotopically idempotent
commutative Smith ideal obtained in Proposition~\ref{theorem-idem}.  Let $U$ denote a right Quillen adjoint of $F$. Then all morphisms in the diagram
\[
        UF(M)  \overset{\varphi}{\rightarrow}   UF \Map_\mathcal{M}(\tilde{I},\, M) \overset{\psi}{\leftarrow}      \Map_\mathcal{M}(\tilde{I},\, M)
\]
are weak equivalences for any fibrant object $M$ of $\mathcal{M}$.
\end{theorem}
\begin{proof}
Let $j: I \to V$. For any fibrant object $M$ of $\mathcal{M}$,
one has a commutative diagram
{\small \[
\xymatrix@1{
\Hom_{\mathrm{Ho}(\mathcal{M})}(V,\, UF(M)   )
\ar[r]^{\hspace{-5mm}\varphi_*}  &   \Hom_{\mathrm{Ho}(\mathcal{M})}(V,\,  UF\Map_\mathcal{M}(I,\, M )  )   & \ar[l]_{\hspace{5mm}\psi_*}  \Hom_{\mathrm{Ho}(\mathcal{M})}(V,\,  \Map_\mathcal{M}(\tilde{I},\, M )  ) \\
\Hom_{\mathrm{Ho}(\mathcal{M})}(F_*(F(Q(V))),\, M   ) \ar[r] \ar[u]   & \Hom_{\mathrm{Ho}(\mathcal{M})}(F_*(F(Q(V))) \otimes I,\,  M  )   \ar[u] & \ar[l] \ar[u] \Hom_{\mathrm{Ho}(\mathcal{M})}(\tilde{I},\, M),
}
\]
}
where all vertical maps are isomorphisms. Since $M$ is fibrant, all of
the lower arrows induced by weak equivalences in Lemma~\ref{unit} are
isomorphisms.  Therefore both $\varphi_*$ and $\psi_*$ are
isomorphisms. Since $V$ generates $\mathcal{M}$, both of the homotopy
cokernels of $\varphi$ and $\psi$ are contractible, implying that they
are weak equivalences. \qed
\end{proof}

\section{Operadic non-unital algebra objects of stable symmetric monoidal model categories}
\label{sec:non-unital}
\subsection{Operadic Smith ideals}
\begin{remark}
The definitions of {\it colors}, {\it symmetric sequences}, and {\it
colored operads} in symmetric monoidal categories can be found in
standard references on operads; see, for example, White--Yau~\cite[Section 3]{WY2024}.

Further, in the case
that the monoidal structure on $\mathcal{M}$ is closed, algebras over
$\mathcal{O}$ are determined by operad maps from $\mathcal{O}$ to
endomorphism operads.
For any $\mathfrak{C}$-colored object $A=\{A_{c} \}_{c \in
\mathfrak{C}}$ of $\mathcal{M}$, the endomorphism operad
$\mathbf{End}_{\mathcal{M}}(A)$ is defined by
\[
  \mathbf{End}_{\mathcal{M}}(A)(c_1,\ldots,c_n;d)=
\Map_\mathcal{M}(A_{c_1}\otimes \cdots \otimes A_{c_n}, A_d).
\]
Giving a map $\theta: \mathcal{O} \to \mathbf{End}_{\mathcal{M}}(A)$ of
$\mathfrak{C}$-colored operads determines an $\mathcal{O}$-algebra
structure on $A$.
\end{remark}
\begin{definition}
\label{Arrow-Operad}
Let $\mathcal{M}$ be a symmetric monoidal closed (model) category and
$\mathcal{O}$ a colored operad in $\mathcal{M}$. Then we write
$\overrightarrow{\mathcal{O}^\Box}=L_1(\mathcal{O})$ and
$\overrightarrow{\mathcal{O}^\otimes}=L_0(\mathcal{O})$. Since $L_1$
(resp. $L_0$) is strict monoidal (See \cite[Definition 3.1.1
(7)]{WY2024}), $\overrightarrow{\mathcal{O}^\Box}$
(resp. $\overrightarrow{\mathcal{O}^\otimes}$) is a colored operad in
$\mathrm{Ar}^\Box(\mathcal{M})$
(resp. $\mathrm{Ar}^\otimes(\mathcal{M})$). Let
$\Alg_{\overrightarrow{\mathcal{O}^\Box}}(\mathrm{Ar}^\Box(\mathcal{M}))$
(resp. $\Alg_{\overrightarrow{\mathcal{O}^\otimes}}(\mathrm{Ar}^\otimes(\mathcal{M}))$)
denote the category of $\mathcal{O}^\Box$-algebra objects of
$\mathrm{Ar}^\Box(\mathcal{M})$ (resp. $\mathcal{O}^\otimes$-algebra
objects of $\mathrm{Ar}^\otimes(\mathcal{M})$).
\end{definition}

Operadic Smith ideals are defined as follows:
\begin{definition}[\cite{WY2024} Definition 3.3.1]
Let $\mathcal{M}$ be a symmetric monoidal closed (model) category and
$\mathcal{O}$ a colored operad in $\mathcal{M}$. Then an
{$\mathcal{O}$-Smith ideal} is an object of
$\Alg_{\overrightarrow{\mathcal{O}^\Box}}(\mathrm{Ar}^\Box(\mathcal{M}))$.
\end{definition}

\subsection{Non-unital algebras from the (operadic) Smith ideal theory viewpoint}
We briefly review non-unital rings, which are rings not assumed to have a multiplicative unit $1$. Any unital
ring is regarded as a non-unital ring by forgetting the existence of
$1$. An {\it augmentation} $\varepsilon_A: A \to V$ is a ring homomorphism such that the unit homomorphism $V \to A$ is a section of $\varepsilon_A$, and a
$V$-algebra with an augmentation is called an {\it augmented
$V$-algebra}. Let $\Alg_{V/ / V}$ denote the category of augmented
$V$-algebras and $\Alg^{\rm nu}_V$ the category of non-unital $V$-algebras. It is well-known that there exists a categorical
equivalence between $\Alg^{\rm nu}_V$ and $\Alg_{V/ / V}$ as follows:
For any  $V$-algebra $A$, the direct sum $V \oplus A$ has a canonical
unital ring structure defined by
\[
 (m,\,a) \cdot (n,\,b) = (mn,\, na+mb+ab)
\]
for $m, \, n \in V$ and $a,\,b \in A$.  It is easily
checked that the adjunction
\[
  V \oplus (-) :  \Alg^{\rm nu}_{V} \rightleftarrows  \Alg_{V/ / V}: \mathrm{Ker} (\epsilon_{(-)} : - \to V)
\]
is a pair of categorical equivalences, where $\mathrm{Ker} (\varepsilon_{(-)}: - \to V)$ is the kernel functor of augmentations.

From the Smith ideal theory viewpoint, the above categorical equivalence is a special case of cokernel-and-kernel correspondence between the arrow categories. We consider the case of abelian categories. Let $\mathcal{M}$ be a symmetric closed monoidal abelian category and $\mathrm{Ar}^{\rm im}(\mathcal{M})$ denote the full subcategory spanned by those objects $f:R \to S$ such that the unit $f \to \mathrm{ker} (
\mathrm{cok}(f))$ is an isomorphism. Then the cokernel functor
$\mathrm{Ar}^{\rm im}(\mathcal{M}) \to \mathrm{Ar}^{\rm
coim}(\mathcal{M})$ is a categorical equivalence.
\begin{proposition}
\label{stabilization} Let $\mathcal{M}$ be a locally presentable abelian
symmetric monoidal category. Then the arrow category
$\mathrm{Ar}(\mathcal{M})$ is also locally presentable and there exist
reflective localization functors $L_{\rm im}:
\mathrm{Ar}^\Box(\mathcal{M}) \to \mathrm{Ar}^{\rm im}(\mathcal{M}) $ by
the unit $\mathrm{id} \to \ker \circ \mathrm{coker} $ and $L_{\rm coim}:
\mathrm{Ar}^\otimes(\mathcal{M}) \to \mathrm{Ar}^{\rm coim}(\mathcal{M})
$ by the counit $\mathrm{coker} \circ \ker \to \mathrm{id}$ such that the adjunction
\[
\mathrm{cok} :
\mathrm{Ar}^\Box(\mathcal{M}) \rightleftarrows \mathrm{Ar}^\otimes(\mathcal{M}):
\mathrm{ker}
\]
induces categorical equivalences
 \[
  \mathrm{cok} : \mathrm{Ar}^{\rm im}(\mathcal{M})
 \rightleftarrows \mathrm{Ar}^{\rm coim}(\mathcal{M}): \mathrm{ker}.
\]
\end{proposition}
\begin{proof}
Since any abelian category is binormal, $f \to \mathrm{im}(f)$ is an
isomorphism if and only if $f$ is a monomorphism. By
Ad{\'a}mek--Rosick{\'y}~\cite[p.44, Corollary 1.5.4]{LocPresentable},
the arrow category is locally presentable, and it admits reflective
localization. By the definitions of those functors: $\mathrm{im}=
\mathrm{ker} \circ \mathrm{cok}$ and $\mathrm{coim}= \mathrm{cok} \circ
\mathrm{ker}$, the restrictions of $\mathrm{cok}$ and $\mathrm{ker}$ to $\mathrm{Ar}^{\rm
im}(\mathcal{M})$ and $\mathrm{Ar}^{\rm coim}(\mathcal{M})$ are
quasi-inverse functors. \qed
\end{proof}

For any pointed symmetric monoidal closed category $\mathcal{M}$ with a
monoidal unit $V$, let $\Alg_\mathcal{O}(\mathcal{M})_{V//V}$ denote the
full subcategory of
$\Alg_{\mathcal{O}^\otimes}(\mathrm{Ar}^\otimes(\mathcal{M}))$ spanned by
augmentations of $\mathcal{O}^\otimes$-algebra objects.
\begin{definition}
 Let $\mathcal{M}$ be a locally presentable symmetric monoidal abelian
 category with a monoidal unit $V$. A commutative non-unital monoid
 object of $\mathcal{M}$ is a commutative Smith ideal $j:I \to R$ in the
 localized full subcategory $\mathrm{Ar}^{\rm im}(\mathcal{M})$
 satisfying that the cokernel $\mathrm{cok}(j): R \to \mathrm{Coker}(j)
 $ is isomorphic to an augmentation $\varepsilon_R: R \to V$, and
 $\Alg_{\mathcal{O}}^{\rm nu}(\mathcal{M})$ denotes the full subcategory
 of $\Alg_{\mathcal{O}^\Box}(\mathrm{Ar}^\Box(\mathcal{M}))$ spanned by
 non-unital $\mathcal{O}^\Box$-algebra objects. In particular, $\CAlg^{\rm nu}(\mathcal{M})$ denotes the full subcategory $\mathrm{CAlg}(\mathrm{Ar}^\Box(\mathcal{M}))$ spanned by commutative Smith ideals whose cokernels are isomorphic to $V$. We say that an object of $\CAlg^{\rm nu}(\mathcal{M})$ is a {\it non-unital commutative $V$-algebra object} of $\mathcal{M}$.
\end{definition}

\begin{proposition}
\label{non-unital}
 Let $\mathcal{M}$ be a locally presentable symmetric monoidal abelian
 category with a monoidal unit $V$. The categorical equivalence
\[
  \mathrm{cok} : \mathrm{Ar}^{\rm im}(\mathcal{M})
 \to \mathrm{Ar}^{\rm coim}(\mathcal{M})
\]
induces a categorical equivalence
\[
 \Alg^{\rm nu}_{\mathcal{O}^\Box}(\mathcal{M})
\to \Alg_{\mathcal{O}^\otimes}(\mathcal{M})_{V//V}.
\]
\end{proposition}
\begin{proof}
By definition of the category $\Alg_{\mathcal{O}^\Box}^{\rm
nu}(\mathcal{M})$, one has an isomorphism $V \oplus I \to R$. The
assertion is clear. \qed
\end{proof}
\subsection{Definition of operadic non-unital algebra objects of symmetric monoidal categories}
We generalize the categorical equivalence in Proposition~\ref{non-unital}. The main work of this section is to give a suitable definition of non-unital $\mathcal{O}$-algebra objects, which is Definition~\ref{Operadic-non-unital}.

By the notation in Definition~\ref{Arrow-Operad}, we establish a definition of operadic non-unital algebra objects of symmetric monoidal model categories.
\begin{definition}
\label{Operadic-non-unital}
Let $\mathcal{M}$ be a symmetric monoidal closed model category with a monoidal unit $V$. An {\it homotopically augmented
$\mathcal{O}$-algebra over $V$} $f: A \to B$ is an object of
$\Alg_{\overrightarrow{\mathcal{O}^\otimes}}(\mathrm{Ar}^\otimes(\mathcal{M}))$
such that the unit $U^0(V)=(\mathrm{id}_V: V \to V) \to (f: A \to B)$ has a
homotopy section. Let $
\Alg_{\mathcal{O}^\otimes}(\mathrm{Ar}^\otimes(\mathcal{M}))_{V//V}$
denote the semi-model category of homotopically augmented
$V$-algebras. In particular, $\CAlg( \mathcal{M})_{V//V}$ denotes the
semi-model category of homotopically augmented $V$-algebras.

A {\it non-unital $\mathcal{O}$-algebra over $V$} is a
$\mathcal{O}$-Smith ideal whose cokernel is a homotopically augmented
$\mathcal{O}$-algebra.  Let $\Alg_{\mathcal{O}}^{\rm nu}(\mathcal{M})$
denote the full subcategory of
$\Alg_{\mathcal{O}^\Box}(\mathrm{Ar}^\Box(\mathcal{M}))$ spanned by
non-unital $\mathcal{O}^\Box$-algebra objects. In particular,
$\CAlg^{\rm nu}(\mathcal{M})$ denotes the semi-model category of {\it
non-unital commutative $V$-algebra object} of $\mathcal{M}$.
\end{definition}

By the following theorem of White and Yau, we have a Quillen equivalence between semi-model categories of non-unital $\mathcal{O}$-algebras and homotopically augmented $\mathcal{O}$-algebras.
\begin{theorem}[\cite{WY2024} Theorem 4.4.1]
\label{WY-Smith} Let $\mathcal{M}$ be a stable symmetric monoidal model
category with a monoidal unit $V$.
Assume that the adjunction
$\mathrm{cok}: \mathrm{Ar}^\Box(\mathcal{M}) \rightleftarrows
\mathrm{Ar}^\otimes(\mathcal{M}): \mathrm{ker} $ induces a Quillen
adjunction
\[
 \mathrm{cok}:  \Alg_{\mathcal{O}^\Box}(\mathrm{Ar}^\Box(\mathcal{M})) \rightleftarrows \Alg_{\mathcal{O}^\otimes}(\mathrm{Ar}^\otimes(\mathcal{M})) : \mathrm{ker}
\]
between semi-model categories, and the colored operad $\mathcal{O}$ in
 $\mathcal{M}$ is cofibrant in $
 \mathrm{Ar}^\Box(\mathcal{M})^{\mathfrak{C}}$.  Then the Quillen
 adjunction is a Quillen equivalence. \qed
\end{theorem}
\begin{proposition}
\label{non-unital-main} Under the assumption of Theorem~\ref{WY-Smith},
the Quillen equivalence in Theorem~\ref{WY-Smith} induces a Quillen equivalence
\[
 \mathrm{cok}:  \Alg_{\mathcal{O}^\Box}^{\rm nu}(\mathrm{Ar}^\Box(\mathcal{M}))  \rightleftarrows \Alg_{\mathcal{O}^\otimes}(\mathrm{Ar}^\otimes(\mathcal{M}))_{V//V} : \mathrm{ker}.
\]\qed
\end{proposition}

White introduced the commutative monoid axiom of monoidal model categories.
\begin{definition}[\cite{WH2017}, Definition 3.1 and 3.4]
A symmetric monoidal model category $\mathcal{M}$ is said to satisfy the
{\it commutative monoid axiom} if, for all $n >0$, the $n$-fold
symmetric pushout product functor $(-)^{\Box n}/\Sigma_n: \mathcal{M}
\to \mathcal{M} $ preserves all trivial cofibrations. Further, if the
$n$-fold symmetric pushout product functor preserves all cofibrations
and trivial cofibrations, $\mathcal{M}$ is said to satisfy the
{\it strong commutative monoid axiom}.
\end{definition}
In particular, by \cite[Corollary 5.1.2]{WY2024}, we obtain the following:
\begin{corollary}
\label{non-unital-cor}
Let $\mathcal{M}$ be a cofibrantly generated stable model category
satisfying the strong commutative monoid axiom and the monoid axiom, and
assume that the category of commutative monoids in
$\mathrm{Ar}^\Box(\mathcal{M})$ has the transferred model structure.
Then
the cokernel-and-kernel adjunction induces a Quillen equivalence
\[
  \mathrm{cok}:  \CAlg^{\rm nu}( \mathcal{M})  \rightleftarrows  \CAlg( \mathcal{M})_{V//V}: \mathrm{ker}
\]
between model categories. \qed
\end{corollary}

\begin{remark}
If $\mathcal{M}$ is a locally presentable stable symmetric monoidal
model category, the presentable $\infty$-category of augmented algebra
is represented by $\CAlg(\mathcal{M})_{V//V}$. Hence
the left-hand side $\CAlg^{\rm nu}(\mathcal{M})$ is equivalent to the
$\infty$-category defined by Lurie~\cite[p.949, Definition 5.4.4.9 and
Proposition 5.4.4.10]{HA}.
\end{remark}

\subsubsection*{Acknowledgements}  I would like to thank anonymous referees for giving me innovative ideas and constructive comments to use Hovey's Smith ideal theory and White and Yau's operadic Smith ideal theory. I am also grateful to Isamu Iwanari for pointing out the derived-completion example.

\paragraph{Use of AI tools.}
The author used Google’s Gemini-pro 3.1 as a conversational research aid for brainstorming and refining mathematical formulations, and OpenAI’s Prism (GPT-5.2) for editorial and expository assistance, including structural refinement. These tools were used to improve the manuscript’s clarity and to indicate where abbreviated arguments warranted fuller exposition. All results, proofs, and final formulations were reviewed, verified, and approved by the author, who bears sole responsibility for the content of this work.

\bibliographystyle{amsalpha}
\bibliography{bibkato-arxiv}

\nocite{Hirschhorn}
\nocite{HT}
\nocite{kato2023non-unital}
\end{document}